\documentclass{amsart}
\usepackage[utf8]{inputenc}

\title{On subspaces whose weak$^*$ derived sets are proper and norm dense}
\author[Z. Silber]{Zdeněk Silber}
\email{zdesil@seznam.cz}
\address{Department of Mathematical Analysis, Faculty of Mathematics and Physics, Charles University, Sokolovská 83, 186 75, Praha 8, Czech Republic}
\thanks{The research was supported by the Charles University, project GAUK no. 154121 and the grant SVV-2020-260583.}

\usepackage{amsthm}
\usepackage{amsmath}
\usepackage{mathrsfs}
\usepackage{mathtools}
\usepackage[shortlabels]{enumitem}
\usepackage{verbatim}
\usepackage{amsfonts}
\usepackage{amssymb}
\usepackage{graphicx}
\usepackage[a-2u]{pdfx}
\usepackage{tikz}
\usepackage[most]{tcolorbox}
\usepackage{esint}
\usepackage{enumitem}
\usetikzlibrary{matrix,arrows}
\graphicspath{{images/}}

\newtheorem{theorem}{Theorem}[section]
\newtheorem{lemma}[theorem]{Lemma}
\newtheorem{prop}[theorem]{Proposition}
\newtheorem{corollary}[theorem]{Corollary}

\newtheorem{question}{Question}

\theoremstyle{definition}
\newtheorem*{definition}{Definition}
\newtheorem*{remark}{Remark}

\newtheorem{notation}[theorem]{Notation}
\newtheorem{construction}[theorem]{Construction}

\newcommand{\norm}[1]{\left\lVert#1\right\rVert}

\begin{document}

\begin{abstract}
    We study long chains of iterated weak$^*$ derived sets, that is sets of all weak$^*$ limits of bounded nets, of subspaces with the additional property that the penultimate weak$^*$ derived set is a proper norm dense subspace of the dual. We extend the result of Ostrovskii and show, that in the dual of any non-quasi-reflexive Banach space containing an infinite-dimensional subspace with separable dual, we can find for any countable successor ordinal $\alpha$ a subspace, whose weak$^*$ derived set of order $\alpha$ is proper and norm dense.
\end{abstract}

\maketitle

\section{Introduction}

The \textit{weak$^*$ derived set} of a subset $A$ of a dual space $X^*$ is defined as the set of all weak$^*$ limits of bounded nets from $A$, i.e.
\begin{equation*}
    A^{(1)} = \bigcup_{n=1}^\infty \overline{A \cap n B_{X^*}}^{w^*}.
\end{equation*}
If $X$ is separable, bounded sets of $X^*$ are weak$^*$ metrizable, and thus the weak$^*$ derived set $A^{(1)}$ coincides with the \textit{weak$^*$ sequential closure} of $A$, that is the set of all weak$^*$ limits of sequences from $A$. The study of  weak$^*$ sequential closures of subspaces in duals of separable spaces was initiated by Banach \cite{Banach1934} and his school in 1930's. Later, weak$^*$ derived sets and weak$^*$ sequential closures found  significant applications. To name a few, they were applied by Piatetski-Shapiro \cite{piatetski1952} to characterization of sets of uniqueness in harmonic analysis, used by Saint-Raymond \cite{SaintRaymond1976} for Borel and Baire classification of inverses of continuous injective linear operators (see also \cite{Raja2004} for application for non-separable spaces), by Dierolf and Moscatelli \cite{DierolfMoscatelli1987} in the structure theory of Fréchet spaces, or by Plichko \cite{Plichko1986} to solve a problem on universal Markushevich bases posed by Kalton. For additional information and a historical account, see the survey on weak$^*$ sequential closures by Ostrovskii \cite{ostro2001} and the introduction to his new paper \cite{Ostro2022}.

The weak$^*$ derived set needs not to be closed under taking weak$^*$ limits of bounded nets, that is $A^{(1)}$ can be a proper subset of $\left( A^{(1)} \right)^{(1)}$, even if $A$ is a subspace. This inspires the definition of weak$^*$ derived sets of higher order. We use the convention that $A^{(0)} = A$. For a successor ordinal $\alpha$, the weak$^*$ derived set of $A$ of order $\alpha$ is $A^{(\alpha)} =\left(A^{(\alpha - 1)}\right)^{(1)}$. For a limit ordinal $\alpha$ we define $A^{(\alpha)} = \bigcup_{\beta < \alpha} A^{(\beta)}$. The order of $A$ is defined to be the least ordinal $\alpha$, such that $A^{(\alpha)} = A^{(\alpha + 1)}$. Note that it follows from the Krein-Šmulyan theorem that if $A$ is convex, then $A = A^{(1)}$ if and only if $A$ is weak$^*$ closed. Hence, if $A$ is convex, the order of $A$ is the least ordinal $\alpha$ such that $A^{(\alpha)} = \overline{A}^{w^*}$.

It is readily proved that a subspace $A$ of $X^*$ is norming, if and only if $A^{(1)} = X^*$. Recall that $A$ is said to be \textit{norming} if
\begin{align*}
    q_A(x) = \sup \{|f(x)|: \; f \in A \cap B_{X^*}\}
\end{align*}
defines an equivalent norm on $X$. Davis and Lindenstrauss \cite{DavisLinden} have shown that a Banach space is quasi-reflexive if and only if every subspace of its dual which separates points is also norming. Recall that a Banach space $X$ is \textit{quasi-reflexive}, if it is of finite codimension in its bidual. It thus follows by a quotient argument that a Banach space $X$ is quasi-reflexive if and only if $A^{(1)} = \overline{A}^{w^*}$ for every subspace $A$ of $X^*$, or in other words, if and only if the order of any subspace of $X^*$ is at most one. The study of possible orders of subspaces in duals of separable non-quasi-reflexive spaces was completed by the following theorem of Ostrovskii \cite{Ostrovskii1987}:

\begin{theorem} \label{Theorem:SepNonQR}
    Let $X$ be a separable non-quasi-reflexive Banach space. Then for every countable successor ordinal $\alpha$ there is a subspace $A$ of $X^*$ of order $\alpha$.
\end{theorem}

Further, the orders of subspaces in a dual to a separable Banach space must be countable and cannot be limit, see for example \cite{HumphreysSimpson}. Later, García, Kalenda and Maestre \cite{GarciaKalendaMaestre2010} asked the following questions in their paper about extension problems for holomorphic functions on dual Banach spaces.
\begin{itemize}
    \item Let $X$ be a quasi-reflexive Banach space. Is it true that $A^{(1)} = \overline{A}^{w^*}$ for every convex set $A$ in $X^*$?
    \item For which Banach space $X$ does there exist a subspace $A$ of $X^*$ such that $A^{(1)}$ is a proper norm dense subspace of $X^*$?
\end{itemize}
Both of these questions were solved by Ostrovskii in his paper \cite{Ostrovskii2011}. He showed that $A^{(1)} = \overline{A}^{w^*}$ for every convex subset $A$ of $X^*$, if and only if $X$ is reflexive. This result was later developed by the author \cite{Silber2021Conv} and Ostrovskii \cite{Ostro2022} in the spirit of Theorem \ref{Theorem:SepNonQR}. Let us note that it is still an open problem if the order of a convex set can be a countable limit ordinal. Regarding the second question, Ostrovskii proved the following theorem \cite[Theorem 1]{Ostrovskii2011}:

\begin{theorem} \label{Theorem:NonQRCont}
    The dual Banach space $X^*$ contains a linear subspace $A$ such that $A^{(1)}$ is a proper norm dense subset of $X^*$, if and only if $X$ is a non-quasi-reflexive Banach space containing an infinite-dimensional subspace with separable dual.
\end{theorem}

The aim of this paper is to extend the result of Theorem \ref{Theorem:NonQRCont} for higher ordinals in the spirit of Theorem \ref{Theorem:SepNonQR}. The results will be valid for both real and complex scalars.

We will use the following notation. We write $\mathbb{F}$ for the underlying field $\mathbb{R}$ or $\mathbb{C}$. For a sequence $(x_n)_{n=1}^\infty$ in a Banach space $X$ we denote its closed linear span by $[x_n]_{n=1}^\infty$. Analogically, $[x_n]_{n=1}^N$ will denote the linear span of $(x_n)_{n=1}^N$. Recall that a couple $(\{x_i\}_{i \in I},\{x_i^*\}_{i \in I})$ is called a \textit{biorthoganal system} in $X$ if for all $i,j \in I$ we have $x_j^*(i) = \delta_{i,j}$, where $\delta_{i,j} = 1$ if $i = j$ and $\delta_{i,j} = 0$ otherwise. An indexed family $\{x_i\}_{i \in I}$ in $X$ is said to be \textit{minimal} if there is an indexed family $\{x^*_i\}_{i \in I}$ in $X^*$, such that $(\{x_i\}_{i \in I},\{x_i^*\}_{i \in I})$ forms a biorthoganal system. For a subset $A$ of a Banach space $X$ we denote by $A^\perp$ the \textit{annihilator} of $A$ in $X^*$, that is $A^\perp =  \{x^* \in X^*: \; x^*(a) = 0 \text{ for all } a \in A\} \subseteq X^*$. For a subser $B$ of $X^*$ we denote by $B_\perp$ the \textit{preannihilator} of $B$ in $X$, that is $B_\perp = \bigcap_{b \in B} \operatorname{Ker}b \subseteq X$.

\section{The results}

The statement of the main theorem is the following.

\begin{theorem} \label{Theorem:Main}
    Let $X$ be a Banach space. Then the following are equivalent:
    \begin{enumerate}
        \item $X$ is non-quasi-reflexive and contains an infinite-dimensional subspace with separable dual.
        \item There is a subspace $A$ in $X^*$, such that $A^{(1)}$ is a proper norm dense subspace of $X^*$.
        \item For each countable successor ordinal $\alpha$ there is a subspace $A$ in $X^*$, such that $A^{(\alpha)}$ is a proper norm dense subspace of $X^*$.
    \end{enumerate}
\end{theorem}

The equivalence $\textit{(1)} \iff \textit{(2)}$ follows from Theorem \ref{Theorem:NonQRCont} and clearly $\textit{(3)} \implies \textit{(2)}$. The strategy to prove the implication $\textit{(1)} \implies \textit{(3)}$ is to combine the construction from \cite{Ostrovskii1987}, that is the proof of Theorem \ref{Theorem:SepNonQR}, and the construction from \cite{Ostrovskii2011}, which is the proof of Theorem \ref{Theorem:NonQRCont}. We will first find a subspace $W$ of $X$ spanned by a nice biorthogonal system and find a suitable subspace $K$ in $W^*$. In the following lemma we use the notation $n_k = \frac{k(k+1)}{2}$, $k \in \mathbb{N}$. Then the sequences $\{(n_m + 0)_{m=1}^\infty\} \cup \{(n_m + i - 1)_{m=i-1}^\infty: \; i \geq 2\}$ form a partition of $\mathbb{N}$, as illustrated in the following table.

\begin{align*}
    \begin{array}{cccccc}
        (n_m + 0)_{m=1}^\infty & (n_m + 1)_{m=1}^\infty & (n_m + 2)_{m=2}^\infty & (n_m + 3)_{m=3}^\infty & (n_m + 4)_{m=4}^\infty & \dots \\
        1&2&&&& \\
        3&4&5&&& \\
        6&7&8&9&& \\
        10&11&12&13&14& \\
        \vdots&\vdots&\vdots&\vdots&\vdots&\ddots
    \end{array}
\end{align*}

\begin{lemma} \cite[Lemma 2]{Ostrovskii2011} \label{Lemma:BiorthogonalSystem}
    Let $X$ be a non-quasi-reflexive Banach space containing an infinite-dimensional subspace with separable dual. Then there is a minimal system
    \begin{align*}
        \mathcal{W} = \{x_n\}_{n\in \mathbb{N}} \cup \{u_n\}_{n\in \mathbb{N}}
    \end{align*}
    satisfying:
    \begin{enumerate}[(i)]
        \item $\mathcal{W}$ and its biorthogonal functionals are uniformly bounded.
        \item The sequence $(u_n)_{n \in \mathbb{N}}$ is a shrinking basic sequence.
        \item The sequences $(x_{n_p})_{p=1}^\infty$ and $(x_{n_p+j-1})_{p = j-1}^\infty$, $j \geq 2$, have uniformly bounded partial sums.
        \item The sequence of subspaces
            \begin{align*}
                [x_1,x_2,u_1], [x_3,x_4,x_5,u_2], \dots, [x_{n_p},x_{n_p+1}\dots,x_{n_p+p},u_p],\dots
            \end{align*}
            forms a finite-dimensional decomposition of $W := \overline{\operatorname{span}} \mathcal{W}$.
    \end{enumerate}
\end{lemma}

Note that assertions $(i)-(iii)$ follow from the statement of \cite[Lemma 2]{Ostrovskii2011}. Assertion $(iv)$ is shown within its proof. We will now prove that assertion \textit{(3)} of Theorem \ref{Theorem:Main} is satisfied for the space $W$. For information about shrinking bases see Section 3.2. of \cite{nigel2006topics} or section 1.b. of \cite{LindenstraussTzafririClassicalBSI}. For more information about finite-dimensional decompositions see Section 1.g. od \cite{LindenstraussTzafririClassicalBSI}.

\begin{notation} \label{Notation} We will further use the following notation.

    \begin{enumerate}[1.]
        \item For $p \in \mathbb{N}$ we set $x^1_p = x_{n_p}$. For $j \geq 2$ and $p \in \mathbb{N}$ we set $x^j_p = x_{n_{p+j-2} + j-1}$. In this notation, the sequences $(x^j_p)_{p=1}^\infty$, $j \in \mathbb{N}$, form a partition of the set $\{x_n\}_{n=1}^\infty$ and have uniformly bounded partial sums by Lemma \ref{Lemma:BiorthogonalSystem} \textit{(iii)}.
        \item For $w \in \mathcal{W}$ we shall denote by $\widetilde{w}$ the biorthogonal functional of $w$ (with respect to $\mathcal{W}$). We differ from the usual notation $w^*$ since we use upper indices for some elements of $\mathcal{W}$.
        \item For each $n \in \mathbb{N}$ we fix a weak$^*$ cluster point $f_n$ of the sequence $\left( \sum_{j=1}^k x_j^n \right)_{k=1}^\infty$ in $W^{**}$. Note that its existence is guaranteed by Lemma \ref{Lemma:BiorthogonalSystem} $(iii)$ and 1. and that those elements are also uniformly bounded.
        \item For $n \in \mathbb{N}$ we write $P_n$ for the canonical projection onto
            \begin{align*}
                [x_1,x_2,u_1,x_3,x_4,x_5,u_2,\dots,x_{n_n},x_{n_n+1}\dots,x_{n_n+n},u_n].
            \end{align*}
            The projections $P_n$, $n \in \mathbb{N}$, are uniformly bounded by Lemma \ref{Lemma:BiorthogonalSystem} $(iv)$ and the properties of finite-dimensional decompositions, see the discussion after \cite[Definition 1.g.1.]{LindenstraussTzafririClassicalBSI}. The adjoint $P_n^*$ is then a projection of $W^*$ onto
            \begin{align*}
                [\widetilde{x}_1,\widetilde{x}_2,\widetilde{u}_1,\widetilde{x}_3,\widetilde{x}_4,\widetilde{x}_5,\widetilde{u}_2,\dots,\widetilde{x}_{n_n},\widetilde{x}_{n_n+1}\dots,\widetilde{x}_{n_n+n},\widetilde{u}_n].
            \end{align*}
            These projections satisfy $P_n^*(x^*) \overset{w^*}{\longrightarrow} x^*$ for each $x^* \in W^*$ (again, see the discussion after \cite[Definition 1.g.1.]{LindenstraussTzafririClassicalBSI}).
        \item Let $i,k \in \mathbb{N}$. We say that $g \in W^{**}$ is of type $t(i,k)$, if either
            \begin{itemize}
                \item $i \neq k$ and $g = x^i_j + a f_k$ for some $a \in \mathbb{F}$ and $j \in \mathbb{N}$, or
                \item $g = u_i + a f_k$ for some $a \in \mathbb{F}$.
            \end{itemize}
        Let $A \subseteq{N}$. We say that a vector $g$ of type $t(i,k)$ is compatible with $A \subseteq \mathbb{N}$ if
        \begin{itemize}
            \item $g = x^i_j + af_k$ and $i,k \notin A$ or
            \item $g = u_i + af_k$ and $k \notin A$.
        \end{itemize}
        \item We denote the closed span of $(u_n)_{n \in \mathbb{N}}$ by $U$. Then by Lemma \ref{Lemma:BiorthogonalSystem} \textit{(ii)} the sequence $(\widetilde{u}_n \restriction U)_{n=1}^\infty$ is a basis of $U^*$.
        \item We will say that a vector $z \in W$ is finitely supported if $z \in \operatorname{span} \mathcal{W}$. Similarly, we will say that a vector $z^* \in W^*$ is finitely supported if $z^* \in \operatorname{span} \{\widetilde{w}: \; w \in \mathcal{W}\}$.
    \end{enumerate}
\end{notation}

\begin{construction} \label{Construction}
    For $\alpha < \omega_1$, $A \subseteq \mathbb{N}$ infinite with infinite complement and a vector $z + af_k$ of type $t(i,k)$ compatible with $A$, we define sets $\Omega(\alpha,A,z + a f_k)$ in the following recursive way:
    \begin{itemize}
        \item $\Omega(0,A,z + a f_k) = \{z + a f_k\}$;
        \item If $\alpha > 0$ is a successor ordinal, we split $A$ into infinitely many infinite subsets $(A_n)_{n=0}^\infty$ and take a summable sequence $(a_n)_{n=1}^\infty$ of positive numbers. Let $(p(n))_{n=1}^\infty$ be the increasing enumeration of $A_0$. We set
        \begin{align*}
            \Omega(\alpha,A,z + a f_k) = \{z + a f_k\} \cup \bigcup_{n=1}^\infty \Omega(\alpha - 1,A_n,x^k_n + a_n f_{p(n)}).
        \end{align*}
        \item If $\alpha > 0$ is a countable limit ordinal, we fix an increasing sequence $(\alpha_n)_{n=1}^\infty$ of ordinals such that $\alpha_n \rightarrow \alpha$ and again split $A$ into infinitely many infinite subsets $(A_n)_{n=0}^\infty$ and take a summable sequence $(a_n)_{n=1}^\infty$ of positive numbers. Let $(p(n))_{n=1}^\infty$ be the increasing enumeration of $A_0$. We set
        \begin{align*}
            \Omega(\alpha,A,z + a f_k) = \{z + a f_k\} \cup \bigcup_{n=1}^\infty \Omega(\alpha_n,A_n,x^k_n + a_n f_{p(n)}).
        \end{align*}
        \item We set $K(\alpha,A,z + a f_k) = \left(\Omega(\alpha,A,z + a f_k)\right)_\perp$.
    \end{itemize}
\end{construction}

\begin{remark}
    Note that if we set $\alpha_n = \alpha - 1$ for a successor ordinal $\alpha > 0$ and $n \in \mathbb{N}$, we can cover both cases of successor or limit $\alpha$ by the definition
        \begin{align} \tag{$*$} \label{*}
            \Omega(\alpha,A,z + a f_k) = \{z + a f_k\} \cup \bigcup_{n=1}^\infty \Omega(\alpha_n,A_n,x^k_n + a_n f_{p(n)}).
        \end{align}
    We will use this notation later, if the proofs do not depend on whether $\alpha$ is a successor or limit ordinal.
\end{remark}

\begin{lemma} \label{Lemma:Typy}
    Let $\alpha < \omega_1$, $A \subseteq \mathbb{N}$ infinite with infinite complement and a vector $z + a f_k$ of type $t(i,k)$ compatible with $A$. Then every element of $\Omega(\alpha,A,z + af_k)$ is either $z + af_k$ or an element of type $t(l,m)$ for $l \in A \cup \{k\}$ and $m \in A$.
\end{lemma}

\begin{proof}
    We shall proceed by induction over $\alpha$. If $\alpha = 0$, the only element of $\Omega(0,A,z + af_k)$ is $z + af_k$ and the claim holds. Suppose the claim holds for all $\beta < \alpha$. By Construction \ref{Construction} and the remark following it
    \begin{align*}
        \Omega(\alpha,A,z + a f_k) = \{z + a f_k\} \cup \bigcup_{n=1}^\infty \Omega(\alpha_n,A_n,x^k_n + a_n f_{p(n)}).
    \end{align*}
    It follows that any element of $\Omega(\alpha,A,z + a f_k)$ is either $z + af_k$ or an element of $\Omega(\alpha_n,A_n,x^k_n + a_n f_{p(n)})$ for some $n \in \mathbb{N}$. By the induction hypothesis, for all $n \in \mathbb{N}$, all elements of $\Omega(\alpha_n,A_n,x^k_n + a_n f_{p(n)})$ are either $x^k_n + a_n f_{p(n)}$ or of type $t(l,m)$ where $l \in A_n \cup \{p(n)\} \subseteq A$ and $m \in A_n \subseteq A$. Hence, in either case, they are of type $t(l,m)$ where $l \in A \cup \{k\}$ and $m \in A$.
\end{proof}

\begin{prop} \label{Proposition:Rovnost}
    Let $\alpha < \omega_1$, $A \subseteq \mathbb{N}$ infinite with infinite complement and let $z + a f_k$ be of type $t(i,k)$ compatible with $A$. Then $\left(K(\alpha,A,z + a f_k) \right)^{(\alpha)} \subseteq \operatorname{Ker}(z + a f_k)$.
\end{prop}

\begin{proof}
    We will proceed by induction over $\alpha$. It obviously holds for $\alpha = 0$ as $z + af_k \in \Omega(\alpha,A,z + af_k)$. Let us suppose that the claim holds for all ordinals smaller than $\alpha$. We will prove that $\left(K(\alpha,A,z + a f_k) \right)^{(\beta)} \subseteq \operatorname{Ker}(z + a f_k)$ for all $\beta \leq \alpha$ by induction over $\beta$. Again, the case $\beta = 0$ is clear. Now, suppose it holds for some $\beta < \alpha$. Take $y^* \in \left(K(\alpha,A,z + a f_k) \right)^{(\beta+1)}$. Then there is a sequence $(y_n^*)_{n=1}^\infty$ of elements in $\left(K(\alpha,A,z + a f_k) \right)^{(\beta)}$ which weak$^*$ converges to $y^*$. Recall that by Construction \ref{Construction} there is $n_0 \in \mathbb{N}$ such that $\beta \leq \alpha_n$ for all $n \geq n_0$. Hence,
    \begin{align*}
        \left(K(\alpha,A,z + a f_k) \right)^{(\beta)} &= \left(\operatorname{Ker}(z + a f_k) \cap \bigcap_{n=1}^\infty K(\alpha_n,A_n,x^k_n + a_n f_{p(n)}) \right)^{(\beta)} \\
        &\subseteq \bigcap_{n=n_0}^\infty \left( K(\alpha_n,A_n,x^k_n + a_n f_{p(n)}) \right)^{(\beta)} \\
        &\subseteq \bigcap_{n=n_0}^\infty \operatorname{Ker} (x^k_n + a_n f_{p(n)}),
    \end{align*}
    where the first equality follows from (\ref{*}) and the last inclusion follows from the induction hypothesis for $\beta \leq \alpha_n$ and for $n \geq n_0$. Hence, there is $C>0$ such that for every $n \in \mathbb{N}$ and $j \geq n_0$ we have
    \begin{align} \label{Eq1}
        |y_n^*(x_j^k)| = a_j |f_{p(j)}(y_n^*)| \leq C \: a_j.
    \end{align}
    Indeed, the sequence $(y_n^*)_{n=1}^\infty$ is weak$^*$ convergent and thus bounded, and $(f_{p(j)})_{j=1}^\infty$ is also bounded, see point 3. of Notation \ref{Notation}.
    
    We will show that $f_k(y_n^*) \overset{n}{\longrightarrow} f_k(y^*)$. Fix $\epsilon > 0$. By inequality (\ref{Eq1}) there is $m_0 \geq n_0$ such that
    \begin{align} \label{Eq3}
        \sum_{j=m_0}^\infty |y_n^*(x_j^k)| < \epsilon/8.
    \end{align}
    As $f_k$ is a weak$^*$ cluster point of $\left( \sum_{j=1}^m x^k_j \right)_{m=1}^\infty$, there exists $m_1 > m_0$ such that
    \begin{align} \label{Eq4}
        |f_k(y^*) - \sum_{j=1}^{m_1} y^*(x^k_j)| < \epsilon/4.
    \end{align}
    Let us now show that for any $n \in \mathbb{N}$ it holds that
    \begin{align} \label{Eq5}
        |f_k(y_n^*) - \sum_{j=1}^{m_1} y_n^*(x^k_j)| < \epsilon/4.
    \end{align}
    Indeed, for any $n \in \mathbb{N}$ there is $r_n > m_1$ such that $|f_k(y_n^*) - \sum_{j=1}^{r_n} y_n^*(x^k_j)| < \epsilon/8$. Since $m_1 > m_0$, it follows from (\ref{Eq3}) that
    \begin{align*}
        |f_k(y_n^*) - \sum_{j=1}^{m_1} y_n^*(x^k_j)| &\leq |f_k(y_n^*) - \sum_{j=1}^{r_n} y_n^*(x^k_j)| + |\sum_{j=m_1+1}^{r_n} y_n^*(x^k_j)| \\
        &\leq \epsilon/8 + \sum_{j=m_1 + 1}^\infty |y_n^*(x^k_j)| < \epsilon/8 + \epsilon/8 < \epsilon/4.
    \end{align*}
    It follows easily from the fact that $(y_n^*)_{n=1}^\infty$ is weak$^*$ convergent to $y^*$ that there is $n' \in \mathbb{N}$ such that for all $n > n'$ it holds that
    \begin{align} \label{Eq6}
        |\sum_{j=1}^{m_1} y_n^* (x_j^k) - \sum_{j=1}^{m_1} y^* (x_j^k)| < \epsilon/4.
    \end{align}
    By applying the triangle inequality and (\ref{Eq4}), (\ref{Eq5}) and (\ref{Eq6}), we finally get that for all $n > n'$
    \begin{align*}
        |f_k(y_n^*) - f_k(y^*)| \leq \epsilon/4 + \epsilon/4 + \epsilon/4 < \epsilon.
    \end{align*}
    But this precisely means that indeed $f_k(y_n^*) \overset{n}{\longrightarrow} f_k(y^*)$.
    
    Since $z \in W$ and thus $y_n^*(z) \rightarrow y^*(z)$, it follows that
    \begin{align*}
        (z + a f_k)(y^*) = \lim_n (z + a f_k)(y_n^*) = 0
    \end{align*}
    as $y_n^* \in \operatorname{Ker}(z + a f_k)$ by the induction hypothesis. Hence, $y^* \in \operatorname{Ker}(z + a f_k)$. What is left is the induction step for a limit ordinal $\beta$ which follows easily from the definition of weak$^*$ derived sets for limit ordinals.
\end{proof}

\begin{definition}
    Let $\alpha < \omega_1$, $A \subseteq \mathbb{N}$ infinite with infinite complement, $z + a f_k$ a vector of type $t(i,k)$ compatible with $A$. For $d \in \mathbb{F}$ define
    \begin{align*}
        Q_d(\alpha,A, z + a f_k) = K(\alpha,A,z + a f_k) \cap \left( d \widetilde{z} + \operatorname{span} \{\widetilde{x}^t_s: \; s \in \mathbb{N}, \: t \in A \cup \{k\}\} \right).
    \end{align*}
    That is, $Q_d(\alpha,A, z + a f_k)$ are those elements from $K(\alpha,A,z + a f_k)$ which have finite support in the relevant set and have $d$ as the coefficient at $\widetilde{z}$.
\end{definition}

\begin{lemma} \label{Lemma:RozsBiort}
    Let $k,i,j \in \mathbb{N}$. Then $f_k(\widetilde{x}^i_j) = \delta_{k,i}$ and $f_k(\widetilde{u}_i) = 0$.
\end{lemma}

\begin{proof}
    Note that $f_k(\widetilde{u}_i)$ is a cluster point of the sequence $\left( \widetilde{u}_i \left( \sum_{l=1}^m x^k_l \right) \right)_{m=1}^\infty$. It thus follows from biorthogonality that $f_k(\widetilde{u}_i)$ is a cluster point of a sequence of zeroes, and thus $f_k(\widetilde{u}_i) = 0$. Further, $f_k(\widetilde{x}^i_j)$ is a cluster point of the sequence $\left( \widetilde{x}^i_j \left( \sum_{l=1}^m x^k_l \right) \right)_{m=1}^\infty$, and again by biorthogonality, $\widetilde{x}^i_j \left( \sum_{l=1}^m x^k_l \right) = 1$ if $i=k$ and $m \geq l$, and $\widetilde{x}^i_j \left( \sum_{l=1}^m x^k_l \right) = 0$ otherwise. Hence, if $i \neq k$, $f_k(\widetilde{x}^i_j) = 0$ as it is a cluster point of a sequence of zeroes, and if $i = k$, $f_k(\widetilde{x}^i_j) = 1$ as is is a cluster point of a sequence which eventually equals one.
\end{proof}

\begin{lemma} \label{Lemma:Dosazeni}
    Let $\alpha < \omega_1$, $A \subseteq \mathbb{N}$ infinite with infinite complement, $z + a f_k$ a vector of type $t(i,k)$ compatible with $A$. Let $b \in \mathbb{F}$ and $y^* \in b \widetilde{z} + \operatorname{span} \{\widetilde{x}^t_s: \; s \in \mathbb{N}, \: t \in A \cup \{k\}\}$. Then
    \begin{enumerate}[(a)]
        \item $y^* \in \left( Q_b(\alpha,A,z+af_k) \right)^{(\alpha+1)}$;
        \item If moreover $y^* \in \operatorname{Ker}(z + a f_k)$, then $y^* \in \left( Q_b(\alpha,A,z+af_k) \right)^{(\alpha)}$.
    \end{enumerate}
\end{lemma}

\begin{proof}
    First we will show that \textit{(b)} implies \textit{(a)}. Indeed, if we set for $n \in \mathbb{N}$
    \begin{align*}
        y_n^* = y^* - \frac{1}{a} (z+af_k)(y^*) \widetilde{x}^k_n,
    \end{align*}
    then for each $n \in \mathbb{N}$ clearly $y_n^* \in b \widetilde{z} + \operatorname{span} \{\widetilde{x}^t_s: \; s \in \mathbb{N}, \: t \in A \cup \{k\}\}$. Also, by Lemma \ref{Lemma:RozsBiort}, $(z+af_k)(\widetilde{x}^k_n) = a$, and thus $y_n^* \in \operatorname{Ker}(z + af_k)$. It thus follows from \textit{(b)} that $y_n^* \in \left( Q_b(\alpha,A,z+af_k) \right)^{(\alpha)}$. Further, the sequence $(y_n^*)_{n=1}^\infty$ weak$^*$ converges to $y^*$ as the sequence $(\widetilde{x}^k_n)_{n=1}^\infty$ is weak$^*$ null. Indeed, $(\widetilde{x}^k_n)_{n=1}^\infty$ is bounded and pointwise null (that is converging to zero in the topology generated by $\mathcal{W}$), and hence it is also weak$^*$ null as these topologies coincide on bounded sets. Thus, $y^* \in \left( Q_b(\alpha,A,z+af_k) \right)^{(\alpha+1)}$ and \textit{(a)} is true.
    
    We will prove \textit{(b}) by induction over $\alpha$. The case $\alpha = 0$ is clear as $K(0,A,z+af_k) = \operatorname{Ker}(z+af_k)$, see Construction \ref{Construction}. Let us suppose that both \textit{(a)} and \textit{(b)} hold for all $\beta < \alpha$ and take $y^* \in \operatorname{Ker}(z + a f_k)$ as in the statement of the lemma, that is
    \begin{align*}
        y^* = b \widetilde{z} + \sum_{j=1}^m \left( c_j \widetilde{x}^k_j + v_j^* \right),
    \end{align*}
    where $c_j \in \mathbb{F}$ and $v_j^* \in \operatorname{span}\{\widetilde{x}^t_s: \; s \in \mathbb{N}, \: s \in A_j \cup \{p(j)\}\}$ (recall that $A_j$ and $p(j)$ are defined in Construction \ref{Construction}). It follows from Lemma \ref{Lemma:RozsBiort} and the fact that $y^* \in \operatorname{Ker}(z + a f_k)$ that $\sum_{j=1}^m c_j = -\frac{b}{a}$. Indeed, $y^*(z) = b$ and $af_k(y^*) = a \sum_{j=1}^m c_j$. Recall that by Construction \ref{Construction} and (\ref{*})
    \begin{align*}
        K(\alpha,A,z+ a f_k) = \operatorname{Ker}(z + af_k) \cap \bigcap_{n=1}^\infty K(\alpha_n, A_n, x^k_n +a_n f_{p(n)}).
    \end{align*}
    It follows from \textit{(a)} of the induction hypothesis for $\alpha_j < \alpha$, $j=1,\dots,m$, that
    \begin{align*}
        c_j \widetilde{x}^k_j + v_j^* &\in \left( Q_{c_j}(\alpha_j, A_j, x^k_j +a_j f_{p(j)}) \right)^{(\alpha_j + 1)} \\
        &\subseteq \left( Q_{c_j}(\alpha_j, A_j, x^k_j +a_j f_{p(j)} \right)^{(\alpha)}.
    \end{align*}
    Let us now show that
    \begin{align} \label{Eq2}
        b \widetilde{z} + \sum_{j=1}^m Q_{c_j}(\alpha_j, A_j, x^k_j +a_j f_{p(j)}) \subseteq Q_b (\alpha,A,z+ a f_k).
    \end{align}
    Indeed, let us fix an element of the set on the left-hand side of (\ref{Eq2}),
    \begin{align*}
        w^* = b \widetilde{z} + \sum_{j=1}^m w_j^*,
    \end{align*}
    where $w_j^* \in Q_{c_j}(\alpha_j,A_j,x^k_j+a_j f_{p(j)})$. For later convenience set $w_j^* = 0$ for $j > m$. As $A_0 = \{p(j)\}_{j=1}^\infty$, and thus
    \begin{align*}
        A \cup \{k\} = \{k\} \cup \bigcup_{j=1}^\infty (A_j \cup \{p(j)\}),
    \end{align*}
    we get that $w^* \in b \widetilde{z} + \operatorname{span} \{\widetilde{x}^t_s: \; s \in \mathbb{N}, \: t \in A \cup \{k\}\}$. What is left is to show is that $w^* \in K(\alpha,A,z + af_k) = \left(\Omega(\alpha,A,z + af_k)\right)_{\perp}$. Take any $g \in \Omega(\alpha,A,z + af_k)$. Then by Construction \ref{Construction} either $g = z + af_k$ or $g \in \Omega(\alpha_j,A_j,x^k_j + a_j f_{p(j)})$ for some $j \in \mathbb{N}$. We shall first deal with the second case. By Lemma \ref{Lemma:Typy}, either $g$ is of type $t(l_1,l_2)$, where $l_1 \in A_j \cup \{p(j)\}$ and $l_2 \in A_j$, or $g = x^k_j + a_j f_{p(j)}$. In both cases, as the sets $A_l \cup \{p(l)\}$, $l \in \mathbb{N}$, are disjoint, we get, using Lemma \ref{Lemma:RozsBiort}, that $g(w_l^*) = 0$ for $l \neq j$. Hence, as $w_j^* \in K(\alpha_j,A_j,x^k_j + a_j f_{p(j)})$,
    \begin{align*}
        g(w^*) = g(w_j^*) = 0.
    \end{align*}
    Now we deal with the case $g = z + af_k$. We have that
    \begin{align*}
        g(w^*) = (z + af_k)(w^*) = b + \sum_{j=1}^m ac_j = 0
    \end{align*}
    as $w_j^* \in Q_{c_j}(\alpha_j,A_j,x^k_j+a_j f_{p(j)})$ and $\sum_{j=1}^m c_j = - \frac{b}{a}$. We have shown that for any $g \in \Omega(\alpha,A,z + af_k)$, $g(w^*) = 0$, and thus $w^* \in K(\alpha,A,z + af_k)$. Hence, (\ref{Eq2}) is proved.
    
    Finally,
    \begin{align*}
        y^* &= b \widetilde{z} + \sum_{j=1}^m \left( c_j \widetilde{x}^k_j + v_j^* \right) \in b \widetilde{z} + \sum_{j=1}^m \left( Q_{c_j}(\alpha_j, A_j, x^k_j +a_j f_{p(j)}) \right)^{(\alpha)} \\
        &\subseteq \left( b \widetilde{z} + \sum_{j=1}^m Q_{c_j}(\alpha_j, A_j, x^k_j +a_j f_{p(j)}) \right)^{(\alpha)} \subseteq \left( Q_b(\alpha,A,z + af_k) \right)^{(\alpha)}
    \end{align*}
    and \textit{(b)} is proved.
\end{proof}

\begin{corollary} \label{Corollary:Dosazeni}
    Let $\alpha < \omega_1$, $A \subseteq \mathbb{N}$ be infinite with infinite complement and $z + a f_k$ be a vector of type $t(i,k)$ compatible with $A$. Then $W^* = \left( K(\alpha,A,z + a f_k) \right)^{(\alpha+1)}$. Moreover, any finitely supported $y^* \in \operatorname{Ker}(z + a f_k)$ is an element of $\left( K(\alpha,A,z + a f_k) \right)^{(\alpha)}$.
\end{corollary}

\begin{proof}
    Let us first prove the second statement. Take a finitely supported vector $y^* \in \operatorname{Ker}(z + a f_k)$. As $y^*$ is finitely supported, we have
    \begin{align*}
        y^* = \sum_{w \in \mathcal{W}} y^*(w) \widetilde{w},
    \end{align*}
    where only finitely many of the summands are nonzero. We can thus set
    \begin{align*}
        z^* = \sum_{w \in \mathcal{W} \setminus \left(\{z\} \cup \{x^t_s: \; s \in \mathbb{N}, \: t \in A \cup \{k\}\}\right)} y^*(w) \widetilde{w}.
    \end{align*}
    Then $y^* - z^* \in y^*(z) \widetilde{z} + \operatorname{span} \{\widetilde{x}^t_s: \; s \in \mathbb{N}, \: t \in A \cup \{k\}\}$ and $y^* - z^* \in \operatorname{Ker}(z + a f_k)$ by Lemma \ref{Lemma:RozsBiort}. It follows from Lemma \ref{Lemma:Dosazeni} that $y^* - z^* \in \left( K(\alpha,A,z + a f_k) \right)^{(\alpha)}$. Hence, also $y^* \in \left( K(\alpha,A,z + a f_k) \right)^{(\alpha)}$, as $z^* \in K(\alpha,A, z + a f_k)$ by Lemmata \ref{Lemma:Typy} and \ref{Lemma:RozsBiort}.
    
    Now let us prove the first statement. Take any $y^* \in W^*$ and define for $n \in \mathbb{N}$
    \begin{align*}
        y_n^* = P_n^*(y^*) - \frac{1}{a} (z + a f_k)(P_n^*(y^*)) \widetilde{x}^k_n.
    \end{align*}
    Then each $y_n^*$ is finitely supported and is also an element of $\operatorname{Ker}(z + a f_k)$ as by Lemma \ref{Lemma:RozsBiort} we have that $(z+af_k)(\widetilde{x}^k_n) = a$. Hence, $y_n^*$ is an element of $\left( K(\alpha,A,z + a f_k) \right)^{(\alpha)}$ by the already proved part of the corollary. It thus follows that $y^*$, which is the weak$^*$ limit of the sequence $(y_n^*)_{n=1}^\infty$ as the sequence $(\widetilde{x}^k_n)_{n=1}^\infty$ is weak$^*$ null, is an element of $\left( K(\alpha,A,z + a f_k) \right)^{(\alpha+1)}$.
\end{proof}

\begin{theorem} \label{Theorem:MainW}
    Let $0 < \alpha < \omega_1$ be a successor ordinal and $(a_n)_{n=1}^\infty$ be a summable sequence of positive numbers. Let $(A_n)_{n=0}^\infty$ be a partition of $\mathbb{N}$ into countably many infinite subsets and let $(q(n))_{n=1}^\infty$ be the increasing enumeration of $A_0$. Set
    \begin{align*}
        K = \bigcap_{n=1}^\infty K(\alpha - 1,A_n, u_n + a_n f_{q(n)}).
    \end{align*}
    Then $K^{(\alpha)} \subsetneq \overline{K^{(\alpha)}} = W^{*}$.
\end{theorem}

\begin{proof}
    We will prove the following claims:
    \paragraph*{\textbf{Claim 1}} $K^{(\alpha)} \neq W^*$. Indeed, by Proposition \ref{Proposition:Rovnost}
    \begin{align*}
        K^{(\alpha - 1)} \subseteq \bigcap_{n=1}^\infty \operatorname{Ker}(u_n + a_n f_{q(n)}).
    \end{align*}
    Hence, there is $C>0$ such that any functional $y^* \in K^{(\alpha-1)}$ of norm at most one satisfies for each $n \in \mathbb{N}$
    \begin{align*}
        |y^*(u_n)| = a_n|f_{q(n)}(y^*)| \leq C a_n
    \end{align*}
    (recall that $(f_{q(n)})_{n=1}^\infty$ is bounded). It follows that $K^{(\alpha-1)}$ is not norming. Indeed, if $K^{(\alpha-1)}$ was norming, the sequence $(u_n)_{n=1}^\infty$ would be norm null. But this cannot happen as the biorthogonal functionals $(\widetilde{u}_n)_{n=1}^\infty$ are bounded. Hence, $K^{(\alpha)} = \left( K^{(\alpha-1)} \right)^{(1)} \neq W^*$.
    
    \paragraph*{\textbf{Claim 2}} $\widetilde{u}_n \in K^{(\alpha)}$ for each $n \in \mathbb{N}$. It follows from Lemma \ref{Lemma:Dosazeni} \textit{(a)} that $\widetilde{u}_n \in \left( Q_1(\alpha-1,A_n,u_n + a_nf_{q(n)}) \right)^{(\alpha)}$. Further, it follows from Lemma \ref{Lemma:Typy} and Lemma \ref{Lemma:RozsBiort} that for any $j \neq n$ we have that $Q_1(\alpha-1,A_n,u_n+a_if_{q(n)}) \subseteq K(\alpha-1,A_j,u_j+a_jf_{q(j)})$. Moreover, $Q_1(\alpha-1,A_n,u_n+a_if_{q(n)}) \subseteq K(\alpha-1,A_n,u_n+a_nf_{q(n)})$ by definition. Hence,
    \begin{align*}
        \widetilde{u}_n &\in \left( Q_1(\alpha-1,A_n,u_n + a_nf_{q(n)}) \right)^{(\alpha)} \\
        &\subseteq \left( \bigcap_{j=1}^\infty K(\alpha-1,A_j,u_j + a_jf_{q(j)}) \right)^{(\alpha)} = K^{(\alpha)}.
    \end{align*}
    
    \paragraph*{\textbf{Claim 3}} $U^\perp \subseteq \overline{K^{(\alpha)}}$. Recall that $U = [u_n]_{n=1}^\infty$. Let $y^* \in U^\perp$ and for $n \in \mathbb{N}$ set $y_n^* = P_n^*(y^*)$ and
    \begin{align*}
        z_n^* = y_n^* - \sum_{m=1}^\infty a_m f_{q(m)} (y_n^*) \widetilde{u}_m.
    \end{align*} 
    Then, for each $n \in \mathbb{N}$, the element $y_n^*$ is finitely supported. It follows from Lemma \ref{Lemma:RozsBiort} that for each $n \in \mathbb{N}$ the number $f_{q(m)}(y_n^*)$ is nonzero for only finitely many $m \in \mathbb{N}$. Hence, for all $n \in \mathbb{N}$, $z_n^*$ is also finitely supported. Further, as
    \begin{align*}
        \mathcal{W} = \bigcup_{i \in \mathbb{N}} \{u_i\} \cup \{x^t_s : \: s \in \mathbb{N}, \; t \in A_i \cup \{q(i)\}\},
    \end{align*}
    each $z_n^*$ can be decomposed as $z_n^* = \sum_{i=1}^m w_{n,i}^*$ for some $m \in \mathbb{N}$, where
    \begin{align*}
        w_{n,i}^* \in a_{n,i} \widetilde{u}_i + \operatorname{span} \{\widetilde{x}^t_s : \: s \in \mathbb{N}, \: t \in A_i \cup \{q(i)\}\}
    \end{align*}
    for $a_{n,i} = z_n^*(u_i) = y_n^*(u_i) - a_i f_{q(i)}(y_n^*) = - a_i f_{q(i)}(y_n^*)$. Indeed,
    \begin{align*}
        y_n^*(u_i) = (P_n^* \: y^*)(u_i) = y^* (P_n \: u_i) = 0
    \end{align*}
    as either $P_n u_i = u_i$ if $i \leq n$, or $P_n u_i = 0$ if $i > n$ -- in both cases $y^*(P_n u_i) = 0$ as $y^* \in U^\perp$. Moreover, for $i=1,\dots,m$, we have that $w_{n,i}^*$ is an element of $\operatorname{Ker}(u_i + a_i f_{q(i)})$ as
    \begin{align*}
        (u_i + a_i f_{q(i)}) (w_{n,i}^*) &= (u_i + a_i f_{q(i)}) (z_n^*) = (u_i + a_i f_{q(i)}) (y_n^*) - a_i f_{q(i)}(y_n^*) \\
        &= y_n^*(u_i) = 0,
    \end{align*}
    where the first equality follows from Lemma \ref{Lemma:RozsBiort}. Recall that for $i,j \in \mathbb{N}$ and $c \in \mathbb{F}$ it holds that $Q_c(\alpha-1,A_i,u_i+a_if_{q(i)}) \subseteq K(\alpha-1,A_j,u_j+a_jf_{q(j)})$ by Lemma \ref{Lemma:Typy} and Lemma \ref{Lemma:RozsBiort}. It then follows from Lemma \ref{Lemma:Dosazeni} \textit{(b)} that for $i=1,\dots,m$
    \begin{align*}
        w_{n,i}^* &\in \left(Q_{a_{n,i}}(\alpha-1,A_i,u_i + a_i f_{q(i)})\right)^{(\alpha-1)} \\
        &\subseteq \left( \bigcap_{j=1}^\infty K(\alpha-1,A_j,u_j + a_j f_{q(j)}) \right)^{(\alpha-1)} = K^{(\alpha-1)}.
    \end{align*}
    Hence, $z_n^* \in K^{(\alpha - 1)}$ for each $n \in \mathbb{N}$. It follows from boundedness and a diagonal argument that there is an increasing sequence $(n_k)_{k=1}^\infty$ of integers and a sequence of scalars $(c_m)_{m \in \mathbb{N}}$ such that $a_m f_{q(m)}(y_{n_k}^*) \rightarrow c_m$ for each $m \in \mathbb{N}$. Further, $(c_m)_{m=1}^\infty$ is absolutely summable as $|c_m| \leq \sup_j \|f_{q(j)}\| \sup_k \|y_{n_k}^*\| a_m \leq C a_m$ for some $C>0$ which does not depend on $m$. Then
    \begin{align*}
        z_{n_k}^* \overset{w^*}{\underset{k \rightarrow \infty}{\longrightarrow}} y^* - \sum_{m=1}^\infty c_m \widetilde{u}_m =: z^*.
    \end{align*}
    Hence, $z^* \in K^{(\alpha)}$ and since
    \begin{align*}
        y^* = \lim_{n \rightarrow \infty} \left( z^* + \sum_{m=1}^n c_m \widetilde{u}_m \right)
    \end{align*}
    and $z^* + \sum_{m=1}^n c_m \widetilde{u}_m \in K^{(\alpha)}$ by Claim 2, it follows that $y^* \in \overline{K^{(\alpha)}}$.
    
    \paragraph*{\textbf{Claim 4}} $\overline{K^{(\alpha)}} = W^*$. Take any $y^* \in W^*$ and $\epsilon > 0$. Since $(u_n)_{n=1}^\infty$ is shrinking (by Lemma \ref{Lemma:BiorthogonalSystem} \textit{(ii)}) there is a finite linear combination $u^* = \sum_{j=1}^m \lambda_j \widetilde{u}_j \restriction U \in U^*$, such that $\norm{y^* \restriction U - u^*} \leq \epsilon$. Let $w^*$ be a Hahn-Banach extension of $y^* \restriction U - u^*$ to $W$. That is $\norm{w^*} \leq \epsilon$ and $w^* \restriction U = y^* \restriction U - u^*$. Then
    \begin{align*}
        y^* - w^* - \sum_{j=1}^m \lambda_j \widetilde{u}_j \in U^\perp \subseteq \overline{K^{(\alpha)}}.
    \end{align*}
    by Claim 3. Claim 2 thus yields that $y^* - w^* \in \overline{K^{(\alpha)}}$. Moreover, as $\norm{w^*} \leq \epsilon$, we get $\operatorname{dist}(y^*, \overline{K^{(\alpha)}}) \leq \epsilon$. As $\epsilon$ was arbitrary, we get $y^* \in \overline{K^{(\alpha)}}$.
\end{proof}

Now we are all prepared to prove Theorem \ref{Theorem:Main}.

\begin{proof}[Proof of Theorem \ref{Theorem:Main}]
    The equivalence \textit{(1)} $\iff$ \textit{(2)} follows from \cite[Theorem 1]{Ostrovskii2011} and the implication $\textit{(3)} \implies \textit{(2)}$ is clear. To prove the implication $\textit{(1)} \implies \textit{(3)}$ we fix a successor ordinal $\alpha < \omega_1$ and use Lemma \ref{Lemma:BiorthogonalSystem} and Theorem \ref{Theorem:MainW} to find a subspace $W$ of $X$ and a subspace $K$ of $W^*$, such that $K^{(\alpha)} \subsetneq \overline{K^{(\alpha)}} = W^*$. Let $E: W \rightarrow X$ be the identity embedding. Then $E^*:X^* \rightarrow W^*$ is the restriction map. Set $A = (E^*)^{-1}(K)$. It follows from \cite[Lemma 1]{Ostrovskii1987} and the fact that $E^*$ is onto that $A^{(\alpha)} = (E^*)^{-1}(K^{(\alpha)}) \subsetneq X^*$. As $E^*$ is an open mapping, the preimage of a dense set is dense. Thus $\overline{A^{(\alpha)}} = X^*$.
\end{proof}

Note that in Theorem \ref{Theorem:Main} \textit{(3)} we restrict ourselves only to successor ordinals. The reason lies in the proof of Claim 3 of Theorem \ref{Theorem:MainW}. More specifically, we needed to pass from a general $y^* \in U^\perp$ to elements $y_n^* = P_n^*(y^*)$ with finite support. Following the proof of Claim 3 for limit ordinal $\alpha$, with considering $(\alpha_n)_{n=1}^\infty$ instead of $\alpha-1$ (see Construction \ref{Construction}), we would end with $z_{n_k}^* \in K^{(\alpha_{n_k})}$, and thus $z^* \in \left(\bigcup_{k=1}^\infty K^{(\alpha_{n_k})} \right)^{(1)} \subseteq K^{(\alpha + 1)}$. We would, however, need $z^*$ to be in $K^{(\alpha)}$. The problem for limit ordinals thus remains open:

\begin{question}
    Let $\alpha$ be a limit ordinal and $X$ be a non-quasi-reflexive Banach space containing an infinite-dimensional subspace with separable dual. Is there a subspace $A$ of $X^*$ such that $A^{(\alpha)} \subsetneq \overline{A^{(\alpha)}} = X^*$?
\end{question}

\bibliographystyle{acm}
\bibliography{bibliography}

\end{document}